\documentclass[10pt,reqno]{amsart}
\usepackage{amsmath,amsfonts,amsthm,amssymb}
\usepackage{datetime}
\usepackage[shortlabels]{enumitem}
\usepackage{tikz}
\usepackage{soul}
\usepackage{hyperref}

\newcommand{\R}{\mathbb{R}}
\newcommand{\N}{\mathbb{N}}
\newcommand{\id}{\operatorname{id}}
\newcommand{\ds}{\displaystyle}
\newcommand{\cl}{\overline}

\newcommand{\inter}[1]{{#1}^\circ}
\newcommand{\orbit}{\mathcal{O}}

\newcommand{\inner}[1]{\left\langle #1 \right\rangle}

\newcommand{\G}{\mathcal{G}}

\newcommand{\bs}{\setminus}

\theoremstyle{plain}
\newtheorem{theorem}{Theorem}[section]
\newtheorem{lemma}[theorem]{Lemma}

\newtheorem{corollary}[theorem]{Corollary}

\theoremstyle{remark}
\newtheorem{remark}[theorem]{Remark}

\theoremstyle{definition}
\newtheorem{definition}[theorem]{Definition}
\newtheorem{example}[theorem]{Example}

\numberwithin{equation}{section}

\begin{document}
\title[Convergence in nonlinear Perron-Frobenius theory]{Convergence of iterates in nonlinear Perron-Frobenius theory}
\author[B. Lins]{Brian Lins}
\date{}
\address{Brian Lins, Hampden-Sydney College}
\email{blins@hsc.edu}

\begin{abstract}
Let $C$ be a closed cone with nonempty interior $\inter{C}$ in a Banach space. Let $f:\inter{C} \rightarrow \inter{C}$ be an order-preserving subhomogeneous function with a fixed point in $\inter{C}$. We introduce a condition which guarantees that the iterates $f^k(x)$ converge to a fixed point for all $x \in \inter{C}$. This condition generalizes the notion of type K order-preserving for maps on $\R^n_{>0}$. We also prove that when iterates converge to a fixed point, the rate of convergence is always R-linear in two special cases: for piecewise affine maps and also for order-preserving, homogeneous, analytic, multiplicatively convex functions on $\R^n_{>0}$.  This later category includes the maps associated with the homogeneous eigenvalue problem for nonnegative tensors. 
\end{abstract}

\maketitle
 
\section{Introduction}

Let $X$ be a real Banach space with norm $\| \cdot \|$.  A \emph{closed cone} is a closed convex set $C \subset X$ such that (i) $\lambda C \subset C$ for all $\lambda \ge 0$ and (ii) $C \cap (-C) = \{0\}$. A closed cone $C$ induces the following partial order on $X$.  We say that $x \le_C y$ whenever $y - x \in C$. When the cone is understood, we will write $\le$ instead of $\le_C$. We will also write $x \ll y$ when $y-x$ is in the interior of $C$, which will be denoted $\inter{C}$. Let $D \subseteq X$ be a domain. A function $f:D \rightarrow X$ is \emph{order-preserving} if $f(x) \le f(y)$ whenever $x \le y$ and it is \emph{strictly order-preserving} if $f(x) \ll f(y)$ when $x \le y$. We will say that $f$ is \emph{homogeneous} if $f(tx) = tf(x)$ for all $t > 0$ and $x \in D$ and \emph{subhomogeneous} if $f(tx) \le tf(x)$ for all $t \ge 1$ and $x \in D$. 

Suppose that $C$ is closed cone with nonempty interior $\inter{C}$, and $f: \inter{C} \rightarrow \inter{C}$ is order-preserving and (sub)homogeneous. 
Whether or not an eigenvector (or fixed point) exists in $\inter{C}$ is a delicate question, even for finite dimensional cones. A good general reference for finite dimensional cones is \cite{LemmensNussbaum} and \cite[Section 3]{Nussbaum07} has results for infinite dimensional cones. See also \cite{Lins22} and the references therein for some recent results on the existence and uniqueness of eigenvectors in the cone $\R^n_{>0}$.  For maps that do have eigenvectors in the interior of the cone, we can ask whether the iterates $f^k(x)$ (suitably normalized) converge to an eigenvector.


Jiang \cite{Jiang96}, motivated by a theorem of Kamke, introduced the following definition.  An order-preserving function $f: \R^n_{\ge 0} \rightarrow \R^n_{\ge 0}$ is \emph{type K order-preserving} if $f(x)_i < f(y)_i$ whenever $x \le y$ and $x_i < y_i$. If a type K order-preserving subhomogeneous function $f:\R^n_{\ge 0} \rightarrow \R^n_{\ge 0}$ has a fixed point in $\R^n_{>0}$, then Jiang proved that $f^k(x)$ converges to a fixed point for all $x \in \R^n_{\ge 0}$ \cite[Theorem 2.3]{Jiang96}.  
In section 3, we extend the definition of type K order-preserving to apply to maps on any closed cone in a Banach space.  We prove in Theorem \ref{thm:main} that if $f: \inter{C} \rightarrow \inter{C}$ is type K order-preserving, subhomogeneous, and has a fixed point in $\inter{C}$, then under a relatively mild compactness assumption, the iterates $f^k(x)$ converge to a fixed point in $\inter{C}$ for all $x \in \inter{C}$.  The compactness condition is always satisfied in finite dimensions. 
Note that type K order-preserving is weaker than strict order-preserving on a cone, and does not imply the uniqueness of a fixed point, or in the homogeneous case, of an eigenvector up to scaling. 

In some applications it is important to know how fast $f^k(x)$ converges to a fixed point. If $f$ has a unique fixed point and the spectral radius of the derivative or semiderivative at the fixed point is less than one, then the rate of convergence is known to be linear \cite{AkGaNu14}. Similar results for homogeneous maps with a unique eigenvector up to scaling are also known. In section 4 we show that in two important special cases, we can guarantee that the iterates converge to fixed points at a linear rate (specifically R-linear convergence), even if the map has more than one fixed point.  In Theorem \ref{thm:piecewise} we prove that if $f$ is a piecewise affine nonexpansive map on a convex subset $M$ of a finite dimensional Banach space $X$ and $f^k(x)$ converges to $u \in M$ for some $x \in M$, then the rate of convergence must be linear. Nonexpansive piecewise affine maps are common in applications of nonlinear Perron-Frobenius theory \cite{AlBoGa21,GaGu98,HeidergottOldservanderWoude,HuCaPe21, Kohlberg80}, so this result is noteworthy. Then in Theorem \ref{thm:anal} we prove that for any order-preserving, homogeneous, multiplicatively convex, and analytic function $f: \R^n_{>0} \rightarrow \R^n_{>0}$ if $f^k(x)/\|f^k(x)\|$ converges to an eigenvector in $\R^n_{>0}$, then the rate of convergence will be linear.  This class of maps includes the functions associated with the homogeneous eigenproblem for nonnegative tensors (see e.g., \cite{ChPeZh08,FrGaHa13,HuHuQi14,HuQi16,Lim05,Lins22,Qi05,ZhQi12,ZhQiWu13}) and also a large family of functions $\mathcal{M}_+$ constructed from generalized means that was introduced by Nussbaum \cite{Nussbaum86}.  Again, Theorem \ref{thm:anal} applies even when the function has more than one linearly independent eigenvector in $\R^n_{>0}$, so it differs from previously known results about linear convergence of power method iterates for nonnegative tensors that have a unique eigenvector up to scaling \cite{FrGaHa13,HuHuQi14,ZhQi12,ZhQiWu13}.


\section{Preliminaries} \label{sec:prelim}

Let $C$ be a closed cone in a real Banach space $X$. For $x, y \in C$, we define 
$$M(x/y) := \inf \{\beta > 0 : x \le_C \beta y \}$$
and 
$$m(x/y) := \sup \{\alpha > 0 : \alpha y \le_C x \}.$$
Note that $M(y/x) = m(x/y)^{-1}$. If $C$ has nonempty interior, then $(x,y) \mapsto M(x/y)$ is continuous on $X \times \inter{C}$ \cite[Lemma 2.2]{LLNW18}.    

Let $X^*$ denote the dual space of $X$.  The \emph{dual cone} of $C$ is 
$$C^* = \{\phi \in X^* : \phi(x) \ge 0 \text{ for all } x \in C \}.$$  
If $C$ has nonempty interior in $X$, then $C^*$ is a closed cone in $X^*$. 
An alternative formula for $M(x/y)$ (see e.g., \cite[Lemma 2.2]{LLNW18}) is:
\begin{equation} \label{functionals}
M(x/y) = \sup_{\phi \in C^*} \frac{\phi(x)}{\phi(y)}.
\end{equation}

We say that $x$ and $y$ are \emph{comparable} and write $x \sim y$ if $M(x/y) < \infty$ and $m(x/y)> 0$.  Comparability is an equivalence relation, and the equivalence classes are called the \emph{parts} of $C$. If $C$ has nonempty interior, then $\inter{C}$ is a part.  
When $x, y \in C$ are comparable, \emph{Thompson's metric} is 
$$d_T(x,y) := \max \{\log M(x/y), \log M(y/x) \}$$
and \emph{Hilbert's projective metric} is 
$$d_H(x,y) := \log \left(\frac{M(x/y)}{m(x/y)}\right).$$
Thompson's metric is a metric on each part of $C$.  Hilbert's projective metric has the following properties for any comparable $x,y, z \in C$. 
\begin{enumerate} 
\item $d_H(x,y) = 0 \text{ if and only if } y = \lambda x \text{ for some } \lambda > 0$.
\item $d_H(\alpha x, \beta y) = d_H(x,y)  \text{ for all } \alpha, \beta > 0$.
\item $d_H(x,y) = d_H(y,x)$.
\item $d_H(x,z) = d_H(x,y)+d_H(y,z)$.
\end{enumerate}
Note that $d_H$ is a metric on the set $\Sigma := \{ x \in \inter{C} : \|x\|=1 \}$. 
For any comparable $x, y \in C$, we have
\begin{equation} \label{HilbertThompson}
d_H(x,y) \le 2 d_T(x,y).
\end{equation}

In finite dimensions, the topologies induced by $d_T$ on $\inter{C}$ and $d_H$ on $\Sigma$ are equivalent to the topologies inherited from the norm.  This is not always true in infinite dimensions.  It is true, however, if the cone $C$ is normal.  A closed cone $C$ in a Banach space $X$ is \emph{normal} if there is a constant $\kappa > 0$ such that 
$$\|x\| \le \kappa \, \|y\| \text{ whenever } 0 \le_C x \le_C y.$$ 

If $C$ is a normal cone with nonempty interior in $X$, and the closed ball $N_R(u) := \{x \in X : \|x-u\| \le R\}$ is contained in $\inter{C}$, then \cite[Proposition 1.3 and Remark 1.4]{Nussbaum88} imply that there is a constant $c > 0$ such that 
\begin{equation} \label{normalThompson}
c^{-1} \|x-u\| \le d_T(x,u) \le c \|x-u\|
\end{equation}
for all $x \in N_R(u)$. 

Let $(M,d)$ be a metric space.  A function $f: M \rightarrow M$ is \emph{nonexpansive} with respect to the metric $d$ if $d(f(x),f(y)) \le d(x,y)$ for all $x, y \in M$. If $C$ has nonempty interior and $f: \inter{C} \rightarrow \inter{C}$ is order-preserving and subhomogeneous, then $f$ is nonexpansive with respect to Thompson's metric \cite[Lemma 2.1.7]{LemmensNussbaum}. If $f$ is order-preserving and homogeneous, then the map $g(x) = f(x)/\|x\|$ is nonexpansive with respect to Hilbert's projective metric on $\Sigma$ \cite[Proposition 1.5]{Nussbaum88}. In that case, $f$ has an eigenvector $u \in \inter{C}$ with $f(u) = \lambda u$ if and only if $g(x)$ has a fixed point. Furthermore, if $f$ has more than one eigenvector $u, v \in \inter{C}$, then the eigenvalues corresponding to $u$ and $v$ must be the same.  This is not necessarily true, however, if $f$ is only subhomogeneous rather than homogeneous.


If $C$ is a normal closed cone with nonempty interior in a Banach space and $f: \inter{C} \rightarrow \inter{C}$ is order-preserving and homogeneous, the \emph{cone spectral radius} of $f$ is 
$$r_C(f) = \limsup_{k \rightarrow \infty} \|f^k(x)\|^{1/k}$$
for some $x \in \inter{C}$. The value of $r_C(f)$ does not depend on the choice of $x$ \cite[Theorem 2.2]{MaNu02}.  Furthermore, since $f(\inter{C}) \subset \inter{C}$, it follows that $r_C(f) > 0$.  If $f$ has an eigenvector in $\inter{C}$, then the corresponding eigenvalue will be $r_C(f)$.  

For any map $f:D \rightarrow D$ on a set $D$, we the \emph{orbit} of a point $x \in D$ under $f$ is the set $\orbit(x,f) := \{f^k(x) : k \in \N \}$.  If $D$ has a topology, then the \emph{omega limit set} of a point $x$ under $f$ is
$$\omega(x,f) := \bigcap_{n \in \N} \overline { \{ f^k(x) : k \ge n \} },$$
where $\overline{A}$ denotes the closure of $A$.

\section{Type K order-preserving maps} \label{sec:typeK}

\begin{definition}
Let $C$ be a closed cone in a Banach space $X$.  Let $D$ be a domain in $X$ and let $f: D \rightarrow X$.  We say that $f$ is \emph{type K order-preserving} if for any $x, y \in D$ with $y \ge x$, there exists $\epsilon > 0$ such that $f(y) - f(x) \ge \epsilon (y-x)$. 
\end{definition}

\begin{theorem} \label{thm:main}
Let $C$ be a closed cone with nonempty interior in a Banach space and let $f:\inter{C} \rightarrow \inter{C}$ be subhomogeneous and type K order-preserving. If $f$ has a fixed point in the interior of $C$ and the closure of the orbit $\mathcal{O}(x,f)$ is compact for some $x \in \inter{C}$, then $f^k(x)$ converges to a fixed point of $f$.  
\end{theorem}

The key insight in the proof of Theorem \ref{thm:main} is that if $\omega(x,f)$ is a compact subset of $\inter{C}$ and $f:\inter{C} \rightarrow \inter{C}$ is nonexpansive with respect to Thompson's metric, then $f$ is an invertible Thompson metric isometry on $\omega(x,f)$ \cite[Lemma 3.1.2 and Corollary 3.1.5]{LemmensNussbaum}.  This result can be traced back to work by Freudenthal and Hurewitz \cite{FrHu36}.

Before proving Theorem \ref{thm:main}, we note the following minor lemma.

\begin{lemma} \label{lem:feps}
Let $C$ be a closed cone with nonempty interior in a Banach space and let $f: \inter{C} \rightarrow \inter{C}$ be subhomogeneous and type K order-preserving. For every $\epsilon > 0$, let $f_\epsilon := f - \epsilon \id.$
Then for any $x,y \in \inter{C}$, there exists $\epsilon > 0$ small enough so that 
$$d_T(f_\epsilon(x),f_\epsilon(y)) \le d_T(x,y).$$
\end{lemma}

\begin{proof}
Let $\beta = \exp d_T(x,y)$. Then
$$x \le \beta y \text{ and } y \le \beta x.$$
Since $f$ is type K order-preserving, there is an $\epsilon > 0$ small enough so that 
$$ \epsilon(\beta y - x)\le f(\beta y) - f(x) \text{ and } \epsilon(\beta x - y) \le f(\beta x) - f(y).$$
Therefore 
$$f_\epsilon(x) \le f_\epsilon(\beta y)  \text{ and } f_\epsilon(y) \le f_\epsilon(\beta x).$$
Note that $f_\epsilon$ is also subhomogeneous, therefore
$$f_\epsilon(x) \le \beta f_\epsilon(y)  \text{ and } f_\epsilon(y) \le \beta f_\epsilon(x).$$
This means that $d_T(f_\epsilon(x),f_\epsilon(y)) \le d_T(x,y).$
\end{proof}

\begin{proof}[Proof of Theorem \ref{thm:main}] Since $f$ has a fixed point in $\inter{C}$, the orbit of $x$ is bounded in Thompson's metric. Therefore, the omega limit set $\omega(x,f)$ is contained in $\inter{C}$. Any subset of $\inter{C}$ which is compact in the norm topology of $\inter{C}$ is also compact in the Thompson metric topology, which may be weaker. So $f$ is an invertible Thompson's metric isometry on $\omega(x,f)$ \cite[Lemma 3.1.2 and Corollary 3.1.5]{LemmensNussbaum}.

Choose $y, z \in \omega(x,f)$ such that $d_T(y,z)$ is maximal.  We may assume without loss of generality that $d_T(y,z) = \log M(z/y)$. Then there is a linear functional $\phi \in C^*$ such that $M(z/y) = {\phi(z)}/{\phi(y)}$ (see \cite[Lemma 2.2]{LLNW18}).  Let $a = \phi(y)$ and $b = \phi(z)$ and note that $b \ge a$ since $0 \le d_T(y,z) = \log(b/a)$. 

Since $f$ is invertible on $\omega(x,f)$, we have $y^{-1} := f^{-1}(y) \in \omega(x,f)$ and $z^{-1} := f^{-1}(z) \in \omega(x,f)$. Since $f$ is an isometry on $\omega(x,f)$,
$$d_T(y^{-1},z^{-1}) = d_T(y,z) = \log\left(\tfrac{b}{a}\right).$$

Let $f_\epsilon = f - \epsilon \id$. For any sufficiently small $\epsilon >0$, $f_\epsilon(y^{-1})$ and $f_\epsilon(z^{-1})$ are both in $\inter{C}$. By Lemma \ref{lem:feps} there is an $\epsilon > 0$ small enough so that 
$$d_T(f_\epsilon(y^{-1}),f_\epsilon(z^{-1})) \le d_T(y^{-1},z^{-1}) = \log\left(\tfrac{b}{a}\right).$$

Observe that $a \le \phi(w) \le b$ for all $ w \in \omega(x,f)$, otherwise $d_T(w,z)$ or $d_T(y,w)$ would be greater than $\log(b/a)$ by \eqref{functionals}, but $\log(b/a)$ is the maximal distance between pairs in $\omega(x,f)$. In particular, $\phi(y^{-1}) \ge a$ and $\phi(z^{-1}) \le b$.   Therefore
\begin{equation} \label{oba}
\phi(f_\epsilon(y^{-1})) = \phi(y) - \epsilon \phi(y^{-1}) \le a(1-\epsilon)
\end{equation}
and
\begin{equation} \label{obb}
\phi(f_\epsilon(z^{-1})) = \phi(z) - \epsilon \phi(z^{-1}) \ge b(1-\epsilon).
\end{equation}
Then by \eqref{functionals}
$$\log\left(\frac{b}{a}\right) \le \log \frac{\phi(f_\epsilon(z^{-1}))}{\phi(f_\epsilon(y^{-1}))} \le d_T(f_\epsilon(y^{-1}),f_\epsilon(z^{-1})) \le \log\left(\frac{b}{a}\right).
$$
We conclude that $\phi(f_\epsilon(y^{-1})) = a(1-\epsilon)$ and $\phi(f_\epsilon(z^{-1})) = b(1-\epsilon)$.  Combined with \eqref{oba} and \eqref{obb}, this means that $\phi(y^{-1}) = a$ and $\phi(z^{-1}) = b$. 

We can repeat this argument to prove that $\phi(z^{-k}) = b$ for all $k \in \N$ where $z^{-k} := f^{-k}(z) \in \omega(x,f)$. However, there is a point $f^m(x) \in \mathcal{O}(x,f)$ that is arbitrarily close to $y$ and an $n \in \N$ such that $f^{m+n}(x)$ is arbitrarily close to $z$.  Then $f^n(y)$ will be arbitrarily close to $z$ by the nonexpansiveness of $f$.  Since $f$ is an isometry on $\omega(x,f)$, we have $d_T(f^n(y),z) = d_T(y,z^{-n})$ arbitrarily small. But since $\phi(y) = a$ and $\phi(z^{-k}) = b$, we have $d_T(y,z^{-k}) \ge \log (b/a)$ which is a contradiction unless $a=b$ and $\omega(x,f)$ is a singleton.  
\end{proof}

\begin{remark}
Unlike \cite[Theorem 2.3]{Jiang96}, we assume that $f$ has a fixed point in $\inter{C}$ in Theorem \ref{thm:main}.
If $f$ does not have a fixed point in $\inter{C}$, then in many circumstances $\omega(x,f)$ will be contained in a convex subset of the boundary of $C$ \cite[Theorem 3.2]{LLNW18}. There is no guarantee, however, that $f$ extends continuously to the boundary of $C$ even in finite dimensions \cite{BuNuSp03}. However, if $C$ is a polyhedral cone, then any order-preserving subhomogeneous map $f: \inter{C} \rightarrow \inter{C}$ does have a continuous extension to all of $C$ that is order-preserving and subhomogeneous \cite[Theorem 5.1.5]{LemmensNussbaum}.  Furthermore, if $C$ is a polyhedral cone and the orbit $\orbit(x,f)$ is bounded in norm, then it is known that the omega limit set $\omega(x,f)$ is a finite set which is a periodic orbit of a single point \cite[Lemma 6.5 and Theorem 6.8]{AGLN06}, even if $\omega(x,f)$ is not contained in $\inter{C}$.  So for polyhedral cones, we can remove the assumption that $f$ has a fixed point in $\inter{C}$ from Theorem \ref{thm:main} as long as $f$ is type K order-preserving on all of $C$. 
\end{remark}

If $C$ is a closed cone in a finite dimensional Banach space, then an order-preserving subhomogeneous map $f:\inter{C} \rightarrow \inter{C}$ has a fixed point in $\inter{C}$ if and only if $\mathcal{O}(x,f)$ is bounded in Thompson's metric for all $x \in \inter{C}$ \cite[Corollary 3.2.5]{LemmensNussbaum}. In that case, the assumption that $\mathcal{O}(x,f)$ has compact closure is automatic.

The conditions of Theorem \ref{thm:main} are more difficult to check in infinite-dimensions.  Note that many important closed cones in infinite dimensional Banach spaces have empty interior.  For example, the cone $L^p([a,b])_{\ge 0}$ of almost everywhere nonnegative functions in $L^p([a,b])$ has empty interior for all $1 \le p < \infty$. However, $L^p([a,b])_{\ge 0}$ is normal since $0 \le f \le g$ implies that $\|f\|_p \le \|g\|_p$. 

Suppose that $C$ is any normal, closed cone in a Banach space $X$. Let $u \in C$ and let $C_u$ be the part of $C$ containing $u$.  For any $x, y \in X$, let 
$$[x,y] := \{ z \in X : x \le z \le y\}.$$ 
Let 
$$X_u := \bigcup_{k > 0} [-ku,ku] ~\text{ and }~ \|x\|_u := \inf \{k > 0 : x \in [-ku,ku]\}.$$ 
Then $(X_u,\|\cdot\|_u)$ is a Banach space which is continuously embedded in $(X,\|\cdot\|)$ \cite[Proposition 19.9]{Deimling}. Furthermore, $\overline{C_u} = C \cap X_u$ is a normal, closed cone with interior equal to $C_u$ in $(X_u,\|\cdot\|_u)$. So we have the following corollary of Theorem \ref{thm:main}. 

\begin{corollary}
Let $C$ be a normal, closed cone in a Banach space $X$, let $u \in C$, and let $C_u$ be the part of $C$ containing $u$.  Suppose that $f: C_u \rightarrow C_u$ is type K order-preserving, subhomogeneous, and has a fixed point in $C_u$.  If $\orbit(x,f)$ has compact closure for some $x \in C_u$, then $f^k(x)$ converges to a fixed point of $f$.  
\end{corollary}

If $C$ is a normal, closed cone with nonempty interior in an infinite dimensional Banach space, and if $f:\inter{C} \rightarrow \inter{C}$ is order-preserving, subhomogeneous, and there is a measure of non-compactness $\gamma$ such that $f$ is $\gamma$-condensing, that is $\gamma(f(B)) < \gamma(B)$ for every bounded subset $B \subset \inter{C}$, then it is also known that $f$ has a fixed point in $\inter{C}$ if and only if $\mathcal{O}(x,f)$ is bounded in Thompson's metric for every $x \in \inter{C}$ \cite[Theorems 4.3 and 4.4]{Nussbaum88}. In that case, the assumption that the closure of $\mathcal{O}(x,f)$ is compact is also automatic since $\orbit (x,f) = \{x\} \cup f(\orbit (x,f))$ which means that $\gamma(\orbit(x,f)) = \gamma(f(\orbit(x,f)))$ and that can only be true if $\gamma(\orbit(x,f)) = 0$. See \cite[Section 7]{Deimling} for more details about measures of non-compactness.  

If $f: \inter{C} \rightarrow \inter{C}$ is order-preserving and subhomogeneous, then $\lambda f + (1-\lambda) \id$ is a type K order-preserving subhomogeneous map with the same fixed points as $f$ for all $0 < \lambda < 1$. Furthermore, if $f$ is $\gamma$-condensing, then so is $f_\lambda$. This follows from the fact that a measure of non-compactness $\gamma$ is a seminorm on the set of bounded subsets of $X$, that is, $\gamma(cB) = |c|\gamma(B)$ and $\gamma(B_1+B_2) \le \gamma(B_1) +\gamma(B_2)$ for all bounded sets $B, B_1, B_2$ and constants $c \in \R$ \cite[Proposition 7.2]{Deimling}. Therefore the following corollary is true. Note the similarity to Krasnoselskii-Mann iteration for norm nonexpansive maps (see e.g., \cite[Section 5]{GaSt20}).  

\begin{corollary} \label{cor:blue1}
Let $C$ be a normal, closed cone with nonempty interior in a Banach space $X$ and let $f:\inter{C} \rightarrow \inter{C}$ be order-preserving and subhomogeneous. Let $\gamma$ be a measure of non-compactness on $X$ and suppose that $f$ is $\gamma$-condensing. Let $f_\lambda = \lambda f + (1- \lambda) \id$ for some $0 < \lambda < 1$. If $f$ has a fixed point in $\inter{C}$, then $f_\lambda^k(x)$ converges to a fixed point of $f$ for every $x \in \inter{C}$. 
\end{corollary}



If $f: \inter{C} \rightarrow \inter{C}$ is order-preserving and homogeneous, we may want to find an eigenvector in the interior that has an eigenvalue not equal to one.  If we don't know the eigenvalue in advance, then we can iterate the normalized map $g(x) = f(x)/\|f(x)\|$.  

\begin{lemma} \label{lem:compactorbits}
Let $C$ be a normal, closed cone with nonempty interior in a Banach space $X$.  Let $f: \inter{C} \rightarrow \inter{C}$ be order-preserving and homogeneous, and suppose that $f$ has an eigenvector in $\inter{C}$.  Let $g(x) = f(x)/\|f(x)\|$ for all $x \in \inter{C}$.  Then $\cl{\orbit(x,r_C(f)^{-1} f)}$ is compact if and only if $\cl{\orbit(x,g)}$ is compact. 
\end{lemma}

\begin{proof} 
For any closed interval $[a,b] \subset \R$ and compact set $K \subset X$, let 
$$[a,b] \cdot K = \{tx : t \in [a,b], x \in K \}.$$ 
Note that $[a,b] \cdot K$ is compact since it is the image of the compact set $[a,b] \times K \subset \R \times X$ under the continuous map $(t,x) \mapsto tx$.  

Let $u \in \inter{C}$ be an eigenvector of $f$ with $\|u \| = 1$.  We may assume without loss of generality that the spectral radius of $f$ is 1 by replacing $f$ with $r_C(f)^{-1} f$.
Then $f(u) = u$, so there are constants $\alpha, \beta > 0$ such that 
$$\alpha u \le f^k(x) \le \beta u$$ 
for all $k \in \N$. 
Since $C$ is a normal cone, there is a constant $\kappa > 0$ such that
$$\alpha \kappa^{-1}  \le \|f^k(x)\| \le \beta \kappa$$ 
for all $k \in \N$.
Therefore 
$$\cl{\orbit(x,f)} \subset [\alpha \kappa^{-1},\beta \kappa] \cdot \cl{\orbit(x,g)}$$
and 
$$\cl{\orbit(x,g)} \subset [\kappa^{-1} \beta^{-1},\kappa \alpha^{-1}] \cdot \cl{\orbit(x,f)}.$$
Thus, if either $\cl{\orbit(x,f)}$ or $\cl{\orbit(x,g)}$ is compact, then so is the other. 
\end{proof}

Combining Lemma \ref{lem:compactorbits} with Theorem \ref{thm:main}, we have the following.

\begin{corollary} \label{cor:mainHomog}
Let $C$ be a normal, closed cone with nonempty interior in a Banach space.  Let $f: \inter{C} \rightarrow \inter{C}$ be Type K order-preserving and homogeneous. Let $g(x) = f(x)/\|f(x)\|$ for all $x \in \inter{C}$. If $f$ has an eigenvector in $\inter{C}$ and $\orbit(x,g)$ has compact closure, then $g^k(x)$ converges to an eigenvector of $f$ for all $x \in \inter{C}$.  
\end{corollary}

\section{Linear convergence to fixed points} \label{sec:linearrate}

Let $(M,d)$ be a metric space and $x^k$ be a sequence in $M$.  We say that $x^k$ converges to $y \in M$ at a \emph{linear rate} if there exist constants $0 < \theta < 1$ and $c > 0$ such that $d(x^k,y) \le c \, \theta^k$ for all $k \in \N$. Note that $x^k$ converges to $y$ at a linear rate if and only if $\limsup_{k \rightarrow \infty} d(x^k,y)^{1/k} < 1$.  This type of convergence is commonly referred to as \emph{R-linear convergence} to distinguish it from the stronger \emph{Q-linear convergence} where there is a constant $0 < \gamma < 1$ such that  $d(x^{k+1},y) \le \gamma \, d(x^k,y)$ for all $k \in \N$.  

The following technical lemma shows that for order-preserving homogeneous maps on a cone, if a sequence of normalized iterates converge to an eigenvector in the interior at a linear rate, then you get the same convergence constant regardless of which metric is used.  

\begin{lemma} \label{lem:technical}
Let $C$ be a normal, closed cone with nonempty interior in a Banach space.  Let $f: \inter{C} \rightarrow \inter{C}$ be order-preserving and homogeneous. Let $x, u \in \inter{C}$ with $\|u\| = 1$, and let $0 < \theta < 1$.  Then the following are equivalent.
\begin{enumerate}[(a)]
\item \label{item:dHlim} $\ds \limsup_{k \rightarrow \infty} d_H(f^k(x),u)^{1/k} \le \theta.$
\item \label{item:dTlim} There is a $\lambda > 0$ such that $\ds \limsup_{k \rightarrow \infty} d_T(f^k(x)/r_C(f)^k,\lambda u)^{1/k} \le \theta.$
\item \label{item:stronglim} There is a $\lambda > 0$ such that $\ds \limsup_{k \rightarrow \infty} \left\|\frac{f^k(x)}{r_C(f)^k} - \lambda u \right\|^{1/k} \le \theta.$ 
\item \label{item:weaklim} $\ds \limsup_{k \rightarrow \infty} \left\| \frac{f^k(x)}{\|f^k(x)\|} - u \right\|^{1/k} \le \theta.$
\end{enumerate}
\end{lemma}

\begin{proof}
We can assume without loss of generality that the cone spectral radius $r_C(f)$ is one by replacing $f$ with $r_C(f)^{-1} f$.  We make this assumption throughout the proof. \\

\noindent
\ref{item:dHlim} $\Rightarrow$ \ref{item:dTlim}. By nonexpansiveness, 
\begin{align*}
d_H(u,f(u)) &= \lim_{k \rightarrow \infty} d_H(f^{k+1}(x),f(u)) \\
&\le \lim_{k \rightarrow \infty} d_H(f^k(x),u) = 0.
\end{align*}
Therefore $u$ is an eigenvector of $f$ and its eigenvector must be $r_C(f) = 1$. Since $f(u) = u$, it follows that $M(f^k(x)/u)$ is a decreasing sequence and $m(f^k(x)/u)$ is an increasing sequence. The fact that 
$$\lim_{k \rightarrow \infty} \log \left( \frac{M(f^k(x)/u)}{m(f^k(x)/u)} \right) = \lim_{k \rightarrow \infty} d_H(f^k(x),u) = 0$$ 
implies that 
$$\lim_{k \rightarrow \infty} M(f^k(x)/u) = \lim_{k \rightarrow \infty} m(f^k(x)/u).$$
Let $\lambda = \lim_{k \rightarrow \infty} M(f^k(x)/u)$. 
Note that $M(f^k(x)/\lambda u) \ge 1$ and $m(f^k(x)/\lambda u) \le 1$ for all $k \in \N$. Then the definitions of $d_H$ and $d_T$ imply that 
$$d_T(f^k(x),\lambda u) \le d_H(f^k(x),\lambda u) ~\text{ for all } k \in \N.$$
Therefore $\limsup_{k \rightarrow \infty} d_T(f^k(x),\lambda u)^{1/k} \le \limsup_{k \rightarrow \infty} d_H(f^k(x),\lambda u)^{1/k} \le \theta.$ \\

\noindent
\ref{item:dTlim} $\Rightarrow$ \ref{item:stronglim}. Let $N_r(u) = \{ x \in X : \|x-u\| \le r \}$ with $r> 0$ small enough so that $N_r(u) \subset \inter{C}$. 
Then there is a constant $c > 0$ such that \eqref{normalThompson} holds for all $x \in N_r(u)$.
Therefore 
$$\| f^k(x) - \lambda u \| \le c d_T(f^k(x),\lambda u) $$
for all $k \in \N$ sufficiently large.  This implies the conclusion.\\

\noindent
\ref{item:stronglim} $\Rightarrow$ \ref{item:weaklim}. This follows from the following inequality.
\begin{align*}
\left\| \frac{f^k(x)}{\|f^k(x)\|} - u \right\| &\le \left\| \frac{f^k(x)}{\|f^k(x)\|} - \frac{f^k(x)}{\lambda} \right\| + \left\| \frac{f^k(x)}{\lambda} - u \right\|  & (\text{triangle ineq.})\\
&= \left| 1 - \frac{\|f^k(x)\|}{\lambda} \right| + \left\| \frac{f^k(x)}{\lambda} - u \right\| \\
&= \left| \|u\| - \frac{\|f^k(x)\|}{\lambda} \right| + \left\| \frac{f^k(x)}{\lambda} - u \right\| \\
&\le \frac{2}{\lambda} \left\| f^k(x) - \lambda u \right\| & (\text{triangle ineq.})
\end{align*} 
for all $k \in \N$. 

\noindent
\ref{item:weaklim} $\Rightarrow$ \ref{item:dHlim}. The following inequality immediately implies the conclusion.
\begin{align*}
d_H(f^k(x), u) &= d_H(f^k(x)/\|f^k(x)\|, u) \\
&\le 2 d_T(f^k(x)/\|f^k(x)\|, u) & (\text{by } \eqref{HilbertThompson})\\
&\le 2 c \left\|\frac{f^k(x)}{\|f^k(x)\|} - u \right\| & (\text{by } \eqref{normalThompson})
\end{align*}
for all $k \in \N$ sufficiently large.
\end{proof}

\subsection{Piecewise affine functions}

Let $M$ be a convex subset of a Banach space $X$.  We say that $f: M \rightarrow M$ is \emph{piecewise affine} if $M$ can be decomposed into a finite union of closed convex sets on each of which $f$ is affine linear. In some special cases, it has been observed that if the iterates of certain nonexpansive piecewise affine maps converge to a fixed point, then they do so at a linear rate. A result of Robinson \cite{Robinson79} implies that this is true for piecewise affine maps on $\R^n$ that are nonexpansive with respect to the Euclidean norm. It is also known for the value iteration operator for solving a Markov decision process \cite{ScFe79} (see also \cite[Appendix 4.A]{Zijm82}).  The following general result appears to be new, however.

\begin{theorem} \label{thm:piecewise}
Let $(X, \|\cdot\|)$ be a finite dimensional Banach space and let $M$ be a convex subset of $X$.  Let $d$ be a metric on $M$ which induces a topology on $M$ that is equivalent to the topology inherited from the norm.  Suppose that $f:M \rightarrow M$ is a piecewise affine map that is nonexpansive with respect to $d$. If for some $x \in M$, the sequence $f^k(x)$ converges to a fixed point $u$, then there are constants $m \in \N$ and $0 < \gamma < 1$ such that 
$$\|f^{k+m}(x) - u \| \le \gamma \|f^k(x)-u\|$$
for all $k \in \N$ sufficiently large.
\end{theorem}

\begin{proof} We will use the following notation for closed balls centered at $u$ in $M$. For any $R > 0$, let 
$$B_R(u) := \{x \in M : d(x,u) \le R\}$$
and let
$$N_R(u) := \{x \in M : \|x-u\| \le R \}.$$
Since $f$ is a piecewise affine and $u$ is a fixed point of $f$, there is a closed ball $B_R(u)$ such that 
$$
f(\lambda x + (1-\lambda) u) = \lambda f(x) + (1-\lambda) u
$$
whenever $x \in B_R(u)$ and $0 < \lambda < 1$. Since $f$ is nonexpansive with respect to $d$, $f(B_R(u)) \subseteq B_R(u)$, so by induction
\begin{equation} \label{locallyhomogeneous2}
f^k(\lambda x + (1-\lambda) u) = \lambda f^k(x) + (1-\lambda) u
\end{equation}
for all $x \in B_R(u)$, $0 < \lambda < 1$, and $k \in \N$.
Since the topology on $(M,d)$ is equivalent to the norm topology, there are constants $0 < r < R$ and $0 < \alpha < \beta$ such that 
$$B_r(u) \subseteq N_\alpha(u) \subset N_\beta(u) \subseteq B_R(u).$$ 
Let $U = \{x \in M : f^k(x) \text{ converges to } u \}$. Let us prove that $U$ is closed. Suppose that $x^n$ is a sequence in $U$ that converges to $x \in M$. Then 
\begin{align*}
\limsup_{k \rightarrow \infty} \|f^k(x) - u\| &= \limsup_{k \rightarrow \infty} \|f^k(x) - f^k(x^n)\| & (\text{since } \lim_{k \rightarrow \infty} f^k(x^n) = u)\\
 &\le \|x - x^n \| & (\text{by nonexpansiveness})
\end{align*}
for any $n$.  Since we can make $\|x-x^n\|$ arbitrarily small, we conclude that $f^k(x)$ converges to $u$, so $x \in U$ and $U$ is closed.  

Let $V_k = \{x \in M : d(f^k(x),u) < r\}$. Each set $V_k$ is open in $M$, and $V_k \subseteq V_{k+1}$ for all $k \in \N$ since $f$ is nonexpansive and $u$ is a fixed point.  Since $X$ is finite dimensional, the set $B_R(u) \cap U$ is compact. The sets $V_k$ are an open cover of $B_R(u) \cap U$, so there is a single $m \in \N$ such that $B_R(u) \cap U \subseteq V_m$.  Thus $f^m(B_R(u) \cap U) \subseteq B_r(u) \cap U$. This implies that $f^m(N_\beta(u) \cap U) \subseteq N_\alpha(u) \cap U$. 

Let $y \in N_\beta(u) \cap U$ with $y \neq u$.  Let $z = u + \lambda^{-1}(y-u)$ where $\lambda = \|y-u\|/\beta$. Then $\|z-u\| = \beta$ and $y = \lambda z + (1-\lambda) u$. By \eqref{locallyhomogeneous2}, 
\begin{align*}
\|f^k(y)- u\| &= \|\lambda f^k(z) + (1-\lambda) u - u\| \\
&= \lambda \|f^k(z) - u\|
\end{align*}
for all $k \in \N$.  Since $\lambda \ne 0$, it follows that $f^k(z)$ converges to $u$ and therefore $z \in N_\beta(u) \cap U$. Because $f^m(N_\beta(u) \cap U) \subseteq N_\alpha(u) \cap U$, we have
\begin{align*}
\|f^m(y)- u\| &= \lambda \|f^m(z) - u\| \\
&\le \lambda \alpha \le \frac{\alpha}{\beta} \|y-u\|.
\end{align*}
Since $f^k(x) \in N_\beta(u)$ for all $k$ sufficiently large, this completes the proof with $\gamma = \frac{\alpha}{\beta}$.  
\end{proof}

Applying Theorem \ref{thm:piecewise} to the metric space $(\inter{C}, d_T)$ gives:
\begin{corollary}
Let $C$ be a closed cone with nonempty interior in a finite dimensional Banach space. Let $f:\inter{C} \rightarrow \inter{C}$ be order-preserving, subhomogeneous, and piecewise affine.  If $f^k(x)$ converges to a fixed point $u \in \inter{C}$ for some $x \in \inter{C}$, then there exists $m \in \N$ and $0 < \gamma < 1$ such that 
$$\|f^{k+m}(x) - u\| \le \gamma \|f^{k}(x)-u\|$$ 
for all $k \in \N$ sufficiently large. 
\end{corollary}

Combining the previous result with Lemma \ref{lem:technical}, we have:
\begin{corollary} \label{cor:piecewiseHomog}
Let $C$ be a closed cone with nonempty interior in a finite dimensional Banach space. Let $f:\inter{C} \rightarrow \inter{C}$ be order-preserving, homogeneous, and piecewise affine.  Let $g(x) = f(x)/\|f(x)\|$ for all $x \in \inter{C}$. If $g^k(x)$ converges to an eigenvector $u \in \inter{C}$ for some $x \in \inter{C}$, then there exists $0 < \theta < 1$ such that 
$$\limsup_{k \rightarrow \infty} \|g^{k}(x) - u\|^{1/k} \le \theta.$$ 
\end{corollary}


A theorem of Kohlberg \cite{Kohlberg80} says that if $X$ is a finite dimensional Banach space and $f: X \rightarrow X$ is piecewise affine and nonexpansive with respect to the norm on $X$, then $f$ has an invariant half-line. Specifically, there exist $v, w \in X$ such that $f(v+tw) = v+(t+1)w$ for all $t \ge 0$.  The vector $w$ is unique for $f$ and is referred to as the \emph{cycle time vector} in some applications.  Note that $f$ has a fixed point if and only if $w=0$.

Suppose that $f$ is a piecewise affine nonexpansive map that does not have a fixed point.  The following lemma says that after a finite number of steps, the dynamics of the iterates of $f$ on $X$ are completely determined by the dynamics of a piecewise affine nonexpansive map that does have a fixed point.  Thus Theorem \ref{thm:piecewise} can be applied to show that if $f^k(x)$ converges to one of the invariant half-lines of $f$, then the rate of convergence is linear. 
\begin{lemma}
Let $X$ be a finite dimensional Banach space.  Let $f: X \rightarrow X$ be piecewise affine and nonexpansive with respect to the norm on $X$. Suppose for $v,w \in X$ that $f(v+tw) = v+(t+1)w$ for all $t \ge 0$.  Let $g(x) = f(x) - w$ for all $x \in X$.  Then for every $x \in X$, there exists $m \in \N$ such that 
$$f^{k+m}(x) = g^k(f^m(x)) + k w$$
for all $k \in \N$.
\end{lemma}

\begin{proof}
Since $f$ is piecewise affine, there is a finite collection of convex subsets $M_i \subset X$ such that $f(x) = A_i x + b_i$ for all $x \in M_i$ where $A_i$ is a linear transformation on $X$ and $b_i \in X$.  
Fix $R > 0$ and let $B_R(y) = \{x \in X: \|x-y\| \le R \}$. Suppose that $M_i \cap B_R(v+tw) \ne \varnothing$ for all $t \ge 0$ large enough.  Then we claim that $x+tw \in M_i$ whenever $x \in M_i$ and $t\ge 0$.  Otherwise, by the Hahn-Banach theorem there is a linear functional $\phi \in X^*$ and a constant $c \in \R$ such that $\phi(y) \ge c$ for all $y \in M_i$, and $\phi(x+tw) < c$ for some $t>0$.  Then $\phi(w) < 0$.  Observe that for any $y \in B_R(v+tw)$, 
\begin{align*}
\phi(y) &= \phi(v) + t \phi(w) + \phi(y-v+tw)\\
&\le \phi(v) + t \phi(w) + R\|\phi\|.
\end{align*}
In particular, if $t > 0$ is large enough, then $\phi(y) < c$ for all $y \in M_i \cap B_R(v+tw)$. That would mean that $M_i \cap B_R(v+tw)$ is empty, which is a contradiction.  This proves the claim.  

Now, suppose that $x + tw \in M_i$ for all $t \ge 0$. Then for all $t \ge 0$, we have
\begin{align*}
\|x-v\| &\ge \|f(x+tw) - f(v+tw)\| & (\text{nonexpansiveness})\\
&= \|A_ix + tA_iw + b_i - v - (t+1)w\| \\
&= \|f(x) + t A_iw - f(v) - tw \| \\
&\ge  t\| A_iw - w \| - \|f(x)-f(v) \|. & (\text{triangle ineq.})
\end{align*}
This means that $\|A_i w - w\|$ must be zero and therefore $f(x+tw) = f(x)+tw$ for all $t \ge 0$.  

We have shown that $f(x+tw) = f(x) + tw$ for all $x \in B_R(v+sw)$ and $t \ge 0$ when $s$ is sufficiently large.  In particular, for any such $x$, $f^k(x) = g^k(x) + k w$ for all $k \in \N$, where $g(y) = f(y) - w$ for all $y \in X$.  The conclusion follows by observing that if $\|x-v\| \le R$, then $f^m(x) \in B_R(v+mw)$ for $m \in \N$.  Thus by choosing $m$ large enough, we have $f^{k+m}(x) = g^k(f^m(x)) + k w$ for all $k \in \N$.
\end{proof}


\subsection{Analytic and multiplicatively convex functions}

Let $[n] := \{1, \ldots, n\}$.  Let $e_i$, $i \in [n]$, denote the standard basis vectors in $\R^n$.  Let $\mathbf{1}$ denote the vector in $\R^n$ with all entries equal to one.  We use $\log : \R^n_{>0} \rightarrow \R^n$ and $\exp: \R^n \rightarrow \R^n_{>0}$ to denote the entrywise natural logarithm and exponential functions on $\R^n$.  

Let $D$ be an open convex subset of $\R^n$.  A function $f: D \rightarrow \R^n$ is \emph{analytic} if each entry $f_i$ is a real analytic function, and $f$ is \emph{convex} if each entry $f_i$ is a convex function.  For $f: \R^n_{>0} \rightarrow \R^n_{>0}$ we say that $f$ is \emph{multiplicatively convex} if $\log \circ f \circ \exp$ is a convex function on $\R^n$.   

Order-preserving homogeneous maps $f: \R^n_{>0} \rightarrow \R^n_{>0}$ that are analytic and multiplicatively convex are particularly nice. Many important features of these functions are determined by the \emph{directed graph} $\mathcal{G}(f)$ \emph{associated with} $f$ which has vertices $[n]$ and an arc from $i$ to $j$ when 
$$\lim_{t \rightarrow \infty} f(\exp(te_j))_i = \infty.$$
If $f$ is analytic, it is differentiable and the Jacobian matrix $f'(x)$ at each $x \in \R^n_{>0}$ is a nonnegative matrix.  When $f$ is analytic and multiplicatively convex, the directed graph $\mathcal{G}(f'(x))$ associated with the Jacobian matrix is always the same as $\mathcal{G}(f)$ at every $x \in \R^n_{>0}$ \cite[Lemma 4.9]{Lins22}.   

\begin{lemma} \label{lem:analTypeK} 
Let $f: \R^n_{>0} \rightarrow \R^n_{>0}$ be order-preserving, homogeneous, multiplicatively convex, and analytic.  Then $f$ is type K order-preserving if and only if there is an arc from $i$ to itself in $\G(f)$ for every $i \in [n]$.  
\end{lemma}
\begin{proof}
($\Rightarrow$) Since $f$ is type K order-preserving, $f(\exp(e_i)) \ge f(\mathbf{1})+\epsilon e_i$ for some $\epsilon > 0$. This means that the function $t \mapsto \log(f(\exp(te_i))_i)$ is not constant.  Therefore since $\log(f(\exp(te_i))_i)$ is a convex function which is increasing and not constant for $t > 0$, it follows that $\lim_{t\rightarrow \infty} f(\exp(te_i))_i = \infty$.

($\Leftarrow$) If there is an arc from every $i \in [n]$ to itself in $\G(f)$, then the entries on the main diagonal of the Jacobian matrix $\G(f'(x))$ are all positive for every $x \in \R^n_{>0}$. Since $f$ is analytic, the entries of $f'(x)$ depend continuously on $x$. 
Suppose $x, y \in \R^n_{>0}$ with $x \le y$. Since the closed line segment connecting $x$ to $y$ is a compact set, we can choose an $\epsilon > 0$ such that $f'(\lambda y + (1-\lambda) x) - \epsilon \id$ is a nonnegative matrix for every $0 \le \lambda \le 1$.  This implies that $f(x)-\epsilon x \le f(y) - \epsilon y$ or equivalently, $f(y)-f(x) \ge \epsilon (y-x)$ which proves that $f$ is type K order-preserving.
\end{proof}

%

Let $f: \R^n_{>0} \rightarrow \R^n_{>0}$ be order-preserving and homogeneous. Let $C_1, \ldots, C_m$ denote the strongly connected components of $\G(f)$.  We call these components the \emph{classes} of $f$.  A class is \emph{final} if there are no arcs leaving the class in $\G(f)$.  If $C$ is a final class, and $j \in C$, then $f(x)_j$ only depends on the entries $x_i$ of $x$ where $i \in C$ \cite[Lemma 4.5]{Lins22}. For any subset $J \subseteq [n]$, we let 
$$\R^J = \{x \in \R^n : x_j = 0 \text{ for all } j \notin J\}$$
and
$$\R^J_{>0} = \{x \in \R^J : x_j > 0 \text{ for all } j \in J\}.$$
Let $P_J \in \R^{n \times n}$ be the orthogonal projection onto $\R^J$. Since $f$ extends continuously to $\R^n_{\ge 0}$ and the extension is order-preserving and homogeneous \cite[Corollary 4.6]{BuNuSp03}, we can define the following functions for every $J \subseteq [n]$ 
$$f_J := P_J f  P_J.$$
The \emph{upper Collatz-Wielandt number} $r(f_J)$ is defined by 
$$r(f_J) = \inf_{x \in \R^J_{>0}} \max_{j \in J} \frac{f_J(x)_j}{x_j}.$$ 
If $f_J(\R^J_{>0}) \subseteq \R^J_{>0}$, then $r(f_J)$ is equivalent to the cone spectral radius of $f_J$ defined in section 2 \cite[Theorem 5.6.1]{LemmensNussbaum}. In particular, the cone spectral radius $r_{\R^n_{\ge 0}}(f)$ and the upper Collatz-Wielandt number $r(f)$ are the same, so we will use $r(f)$ to denote both. A class $C$ is called \emph{basic} if $r(f_C) = r(f)$. 

If $f$ is multiplicatively convex and no arcs leave $J$ in $\G(f)$, then $f_J(x)_j = f(x)_j$ for all $x \in \R^n_{>0}$ and $j \in J$ \cite[Lemma 4.6]{Lins22}. Put another way,
\begin{equation} \label{decouple}
f_J = P_J f
\end{equation}
on $\R^n_{\ge 0}$ when $f$ is multiplicatively convex and no arcs leave $J$ in $\G(f)$. In particular this is true for the final classes of $f$. 
Gaubert and Gunawardena \cite[Theorem 2]{GaGu04} proved that if $f:\R^n_{>0} \rightarrow \R^n_{>0}$ is order-preserving and homogeneous and $\G(f)$ is strongly connected, then $f$ has an eigenvector in $\R^n_{>0}$.  Therefore $f_C$ has an eigenvector in $\R^C_{>0}$ with eigenvalue $r(f_C)$ for every final class $C$.

Recently, it was shown that any order-preserving, homogeneous, multiplicatively convex $f: \R^n_{>0} \rightarrow \R^n_{>0}$ has an eigenvector in $\R^n_{>0}$ if its basic and final classes are the same; and if $f$ is also analytic, then the converse is true as well, that is, $f$ has an entrywise positive eigenvector if and only if all of its basic classes are final and vice versa \cite[Theorem 4.4]{Lins22}. This necessary and sufficient condition for the existence of eigenvectors in $\R^n_{>0}$ generalizes equivalent results for nonnegative matrices \cite[Theorem 2.3.10]{BermanPlemmons} and nonnegative tensors \cite[Theorem 5]{HuQi16}.


The main result of this section is the following. 

\begin{theorem} \label{thm:anal}
Let $f:\R^n_{>0} \rightarrow \R^n_{>0}$ be order-preserving, homogeneous, analytic, and multiplicatively convex. Let $g(x) = f(x)/\|f(x)\|$ for all $x \in \R^n_{>0}$. If $f$ is type K order-preserving and $f$ has an eigenvector in $\R^n_{>0}$, then there is a $0 < \theta < 1$ such that for every $x \in \R^n_{>0}$, there is an eigenvector $u \in \R^n_{>0}$ (which may depend on $x$) for which 
$$\limsup_{k \rightarrow \infty} \|g^k(x) - u \|^{1/k} \le \theta.$$
\end{theorem}

Before proving Theorem \ref{thm:anal} we'll need the following lemma. 

\begin{lemma} \label{lem:helper}
Let $(M,d)$ be a metric space, and let $F:M \rightarrow M$ be nonexpansive with respect to $d$. Let $u \in M$ be a fixed point of $F$ and let $(x^m)_{m \in \N}$ be a sequence in $M$ that converges to $u$. If there are constants $\eta, \theta \in (0, 1)$ and $c > 0$ such that 
\begin{equation} \label{helperA}
d(F^k(x),u) \le c \, \eta^k 
\end{equation} 
for all $x$ in a neighborhood $B_R(u)$ of $u$ and all $k \in \N$ and
\begin{equation} \label{helperB}
d(F(x^m),x^{m+1}) \le c \, \theta^m
\end{equation}
for all $m \in \N$, then 
$$\limsup_{k \rightarrow \infty} d(x^k,u) \le \eta^\lambda = \theta^{1-\lambda} < 1 \text{ where }\lambda = \dfrac{\log \theta}{\log \eta + \log \theta}.$$ 
\end{lemma}

\begin{proof}
For any $k, m \in \N$,
\begin{align*}
d(x^{k+m},F^k(x^m)) &\le \sum_{i = 0}^{k-1} d(F^{i}(x^{m+k-i}),F^{i+1}(x^{k+m-i-1})) & (\text{triangle ineq.})\\ 
&\le \sum_{i = 0}^{k-1} d(x^{m+k-i},F(x^{k+m-i-1})) & (\text{nonexpansiveness})\\ 
&\le \sum_{i = 0}^{k-1} c \, \theta^{m+k-i-1} \le ck \, \theta^m. & (\text{by } \eqref{helperB})\\ 
\end{align*}
As long as $m$ is large enough so that $x^m \in B_R(u)$, \eqref{helperA} holds so
\begin{align*}
d(x^{k+m},u) &\le d(x^{k+m},F^k(x^m)) + d(F^k(x^m),u) & (\text{triangle ineq.})\\
&\le ck \, \theta^m + c \, \eta^k  \le ck (\theta^m + \eta^k).
\end{align*}
Let $n = k+m$.  Observe that
$$(\theta^m + \eta^k)^{1/n} \le 2 \max \{ \eta^\lambda, \theta^{1-\lambda} \}$$
where $\lambda = k/n$. Therefore
$$d(x^{n},u)^{1/n} \le (2cn)^{1/n} \max \{ \eta^\lambda, \theta^{1-\lambda}\}$$
for every rational $\lambda$ with denominator $n$ and $m = (1-\lambda)n$ large enough so that $x^{m} \in B_R(u)$.  
From this, we conclude that 
$$\limsup_{n \rightarrow \infty} d(x^n,u)^{1/n} \le \max \{ \eta^\lambda, \theta^{1-\lambda} \}$$
for every $0 \le \lambda < 1$.
Since $\eta^\lambda$ is a decreasing function of $\lambda$ while $\theta^{1-\lambda}$ is increasing, the minimum occurs when $\eta^{\lambda} = \theta^{1-\lambda}$.  This happens when $\lambda = \dfrac{\log \theta}{\log \eta + \log \theta}$.
\end{proof}

%

\begin{proof}[Proof of Theorem \ref{thm:anal}]
Since $f$ has an eigenvector in $\R^n_{>0}$, its basic and final classes are the same \cite[Theorem 4.4]{Lins22}.  We separate the proof into three cases based on the final classes of $f$.  \\

\noindent
{Case I.} Suppose that $\G(f)$ is strongly connected. In this case, $f'(x)$ is irreducible for all $x \in \R^n_{>0}$. Since $f$ is type K order-preserving, Lemma \ref{lem:analTypeK} implies that the entries of $f'(x)$ on the main diagonal are all positive.  This means that $f'(x)$ is primitive \cite[Lemma 8.5.5]{HornJohnson}. In this case, it is known that $f$ has a unique eigenvector $u \in \R^n_{>0}$ with $\|u\| = 1$ \cite[Corollary 6.4.7]{LemmensNussbaum}. A proof that $g^k(x)$ converges to $u$ at a linear rate for all $x \in \R^n_{>0}$ can be found in \cite[Corollary 5.2]{FrGaHa13}. Their result is stated for a special class of polynomial maps, but actually applies to any order-preserving, homogeneous map $f:\R^n_{>0} \rightarrow \R^n_{>0}$ with an eigenvector $u$ where the Jacobian matrix $f'(u)$ exists and is primitive.  They show by a standard linearization argument that 
$$\limsup_{k \rightarrow \infty} \|g^k(x) - u\|^{1/k} \le \frac{\rho_2(f'(u))}{\rho(f'(u))}$$ 
where $\rho$ denotes the spectral radius of $f'(u)$ and $\rho_2$ denotes the second largest eigenvalue in absolute value, which must be strictly less than $\rho(f'(u))$ since $f'(u)$ is primitive. \\

\noindent
{Case II.} Suppose that all classes of $\G(f)$ are final and basic. In this case, the map $f$ decouples and the value of $f(x)_i$ only depends on the entries of $x$ in the same class as $i$ by \eqref{decouple}. We may assume without loss of generality that $r(f) = 1$ by replacing $f$ with $r(f)^{-1} f$ if necessary. Let $x \in \R^n_{>0}$. For any final class $C$, the previous case and Lemma \ref{lem:technical} imply there is a $u_C \in \R^C_{>0}$ and $0 < \theta_C < 1$ such that
$$\limsup_{k \rightarrow \infty} \|f_C^k(x) - u_C \|^{1/k} \le \theta_C.$$
Therefore
$$\limsup_{k \rightarrow \infty} \|f^k(x) - u \|^{1/k} \le \max_{C \in \mathcal{C}} \theta_C$$
where $\mathcal{C}$ is the set of all final classes of $f$ and $u = \sum_{C \in \mathcal{C}} u_C$. Note that each $u_C$ might depend on $x$, but the constants $\theta_C$ do not.
Combined with Lemma \ref{lem:technical}, this proves that
$$\limsup_{k \rightarrow \infty} \left\|g^k(x) - \frac{u}{\|u\|}  \right\|^{1/k} \le \max_{C \in \mathcal{C}} \theta_C.$$ \\

\noindent
{Case III.} Suppose that $f$ has classes that are not final. Let $J$ be the union of the final classes of $f$, and let $I = [n] \bs J$. We again assume that $r(f) = 1$. Fix $x \in \R^n_{>0}$. Since $f$ is type K order-preserving, there is a fixed point $u \in \R^n_{>0}$ such that $f^k(x)$ converges to $u$ by Theorem \ref{thm:main}.  

Let $u_I := P_I u$ and $u_J := P_J u$. 
Since there are no arcs in $\G(f)$ that leave $J$, \eqref{decouple} holds. Therefore
$$f_J(u_J) = P_J f P_J(u) = P_J f (u) = P_J u = u_J,$$
so $u_J$ is a fixed point of $f_J$. By the result of the previous case, there is a constant $0 < \theta_J < 1$ which does not depend on $x$ such that 
$$\limsup_{k \rightarrow \infty} \|f_J^k(x) - u_J\|^{1/k} \le \theta_J.$$
Then by Lemma \ref{lem:technical} we also have
\begin{equation} \label{thetaJlim}
\limsup_{k \rightarrow \infty} d_T(f_J^k(x), u_J)^{1/k} \le \theta_J.
\end{equation}
 
Since $u$ is an eigenvector of the homogeneous map $f$ with eigenvalue 1, it is also an eigenvector of $f'(u)$ with the same eigenvalue. Since $f'(u)$ is a nonnegative matrix with an entrywise positive eigenvector, the basic and final classes of $f'(u)$ must be the same \cite[Theorem 2.3.10]{BermanPlemmons}. Also, the spectral radius of $f'(u)$ is $\rho(f'(u)) = 1$. Since $\G(f'(u)) = \G(f)$, the final classes of $f$ and $f'(u)$ are the same.  Therefore, the spectral radius $\rho(P_I f'(u) P_I)$ is strictly less than one.  

Let $F(y) := f(P_I y + u_J)$ for any $y \in \R^n_{>0}$. Observe that $F$ is order-preserving and subhomogeneous, so it is nonexpansive with respect to Thompson's metric.  Also, $F(u)= u$. By \eqref{decouple}, 
$$P_J F(y) = P_J f(P_I y + u_J) = P_J f P_J(P_Iy + u_J) = u_J$$
for all $y \in \R^n_{>0}$.  Since $F(y) = P_I F(y) + P_J F(y)$, it follows that $(P_J F)'(u) = 0$ and
\begin{align*}
F'(u) &= (P_I F)'(u) + (P_J F)'(u) \\
&= (P_I F)'(u) \\
&= P_I f'(u) P_I. & \text{(chain rule)}
\end{align*}
Therefore $\rho:=\rho(F'(u)) < 1$.  By \cite[Lemma 5.6.10]{HornJohnson}, for any $\epsilon > 0$, there is a norm $\|\cdot\|$ on $\R^n$ such that 
$$\|F'(u)(y-u)\| \le (\rho+\epsilon) \|y-u\|$$
for all $y \in \R^n_{>0}$.
Since $F$ is differentiable, for any $\epsilon > 0$, there is a neighborhood around $u$ such that 
$$\|u + F'(u)(y-u) - F(y) \| \le \epsilon \|y-u\|.$$
By choosing $\epsilon > 0$ small enough, we have $\rho+2\epsilon < 1$ and
$$
\|F(y) - u\| \le (\rho+2\epsilon)\|y-u\|
$$
and therefore
$$\|F^k(y) - u\| \le (\rho+2\epsilon)^k\|y-u\|$$
for all $y$ in a sufficiently small neighborhood $B_r(u) = \{y \in \R^n_{>0} : d_T(y,u) \le r\}$ and $k \in \N$.  Then by \eqref{normalThompson}, there is a $c > 0$ such that
\begin{equation} \label{rhoineq}
d_T(F^k(y),u) \le c\, (\rho+2\epsilon)^k d_T(y,u).
\end{equation}

For any $y \in \R^n_{>0}$, observe that 
\begin{align*} 
d_T(F(y),f(y)) &= d_T(f(y_I + u_J), f(y)) \\
&\le d_T(P_I y + u_J, y) & (\text{nonexpansiveness})\\
&= d_T(u_J, P_J y).
\end{align*}
In particular, 
\begin{align*}
d_T(F(f^k(x)),f^{k+1}(x)) &\le d_T(u_J,P_J f^k(x)) \\
&= d_T(u_J, f_J^k(x)) & (\text{by } \eqref{decouple})
\end{align*}
for all $k \in \N$.
Then, by \eqref{thetaJlim}, we must have 
$$d_T(F(f^k(x)),f^{k+1}(x)) \le c (\theta_J+\epsilon)^k$$
for all $k \in \N$ if we replace $c$ by a sufficiently large constant. Now the conditions of Lemma \ref{lem:helper} are satisfied for the metric space $(\R^n_{>0}, d_T)$ with the map $F$, fixed point $u$, sequence $x^m:=f^m(x)$, and constants $\theta = \theta_J+\epsilon$ and $\eta = \rho+2 \epsilon$ (where $\epsilon > 0$ can be made arbitrarily small). Therefore 
$$\limsup_{k \rightarrow \infty} d_T(f^k(x), u )^{1/k} \le \rho^{\lambda} = \theta_J^{1-\lambda} < 1$$
where $\lambda = \frac{\log \theta_J}{\log \rho+ \log \theta_J}$.  The conclusion of the theorem follows from Lemma \ref{lem:technical}. 
\end{proof}

\begin{remark}
If $f:\R^n_{>0} \rightarrow \R^n_{>0}$ is order-preserving and homogeneous, but not type K order-preserving, then in general the iterates of $g(x) = f(x)/\|x\|$ will converge to a periodic orbit \cite[Theorem 8.1.7]{LemmensNussbaum}.  Suppose $f$ is also multiplicatively convex and real analytic and $f$ has an eigenvector in $\R^n_{>0}$.  Let $p$ be the least common multiple of the cycle lengths in $\G(f)$.  Then $\G(f^p)$ has an arc from $i$ to itself for all $i \in [n]$.  Therefore $f^p$ is type K order-preserving by Lemma \ref{lem:analTypeK}. Therefore Theorem \ref{thm:anal} implies that $g^{kp}(x)$ converges to a fixed point of $g^p$ in $\R^n_{>0}$ at a linear rate as $k \rightarrow \infty$ for every $x \in \R^n_{>0}$.  This means that $g^k(x)$ converges to points in a periodic orbit of $g$ at a linear rate.  
\end{remark}

The following two examples show that both analyticity and convexity are necessary to guarantee the linear rate of convergence in Theorem \ref{thm:anal}.

\begin{example}
It is not enough for $f$ to be order-preserving, homogeneous, and analytic to guarantee a linear rate of convergence to fixed points in the interior of a cone.  Consider $T:\R^2 \rightarrow \R^2$ defined by
$$T(x) = \begin{bmatrix}
\tfrac{1}{2}(x_1+x_2)-\arctan(\tfrac{1}{2}(x_2-x_1)) \\
\tfrac{1}{2}(x_1+x_2)+ \arctan(\tfrac{1}{2}(x_2-x_1))
\end{bmatrix}.$$
It is easy to check that the Jacobian derivative of $T$ is always a nonnegative matrix, so $T$ is order-preserving.  Let $f = \exp \circ T \circ \log$.  Then $f:\R^2_{>0} \rightarrow \R^2_{>0}$ is order-preserving, homogeneous, and analytic.
Furthermore, $\mathbf{1}$ is the unique eigenvector of $f$ in $\R^2_{>0}$ up to scaling.  
Let $x = \begin{bmatrix} \exp(-1) \\ \exp(1) \end{bmatrix}$.  Then 
$$
f^k(x) = \begin{bmatrix} \exp(-\arctan^k(1)) \\ \exp(\arctan^k(1)) \end{bmatrix}$$
for all $k \in \N$.  Note that the sequence $\arctan^k(1)$ converges to zero at a rate that is much slower than linear. Therefore $d_T(f^k(x),\mathbf{1})$ also converges to zero at a sublinear rate.  
\end{example}

\begin{example}
It is also not enough for $f$ to be order-preserving, homogeneous, and multiplicatively convex to guarantee a linear rate of convergence to fixed points in the interior of a cone.  Consider $T:\R^2 \rightarrow \R^2$ defined by 
$$T(x) = \begin{bmatrix}
\max(x_1,x_2-\arctan(x_2-x_1)) \\
\max(x_2,x_1+\arctan(x_2-x_1)) 
\end{bmatrix}.$$
You can verify that $T$ is order-preserving by noting that both partial derivatives of $x_2+\arctan(x_1-x_2)$ are always nonnegative.  Note that $T$ is also convex, since 
$$T(x)  = x+ \begin{bmatrix}
\max(0,x_2-x_1-\arctan(x_2-x_1)) \\
\max(0,x_1-x_2+\arctan(x_2-x_1)) 
\end{bmatrix}$$
and the map $t \mapsto \max(0,t - \arctan(t))$ is convex. If we let $f(x) = \exp \circ T \circ \log$, then $f:\R^2_{>0} \rightarrow \R^2_{>0}$ is order-preserving, homogeneous, and multiplicatively convex.  The only eigenvector of $f$ in $\R^n_{>0}$ up to scaling is $\mathbf{1}$.  Let $x = \begin{bmatrix} e^{-1} \\ 1 \end{bmatrix}$.  Then 
$$
f^k(x) = \begin{bmatrix}
\exp(-\arctan^k(1)) \\
1
\end{bmatrix}$$
for all $k \in \N$. Therefore $d_T(f^k(x), \mathbf{1})$ converges to zero, but at a rate much slower than linear. 
\end{example}

\section{Applications}

We conclude with some applications of the previous results.  

\subsection{The Wasserstein distance map}
Huizing, Cantini, and Peyr\'{e} \cite{HuCaPe21} studied a nonlinear eigenproblem related to optimal transport.  They consider the closed cone 
$$\mathcal{D}_n := \{C \in \R^{n \times n}_{\ge 0} : C^T = C \text{ and } C_{ii} = 0 \text{ for } i \in [n]\}$$
of nonnegative, symmetric, $n$-by-$n$ matrices with main diagonal equal to zero.  This cone has nonempty interior in the space of all real symmetric $n$-by-$n$ matrices with main diagonal equal to zero.  

For two nonnegative vectors $a, b \in \R^n_{\ge 0}$ which both have entries that sum to one, the \emph{Wasserstein cost} between $a$ and $b$ is obtained by solving
$$W_C(a,b) := \min_{P \in \Pi_{a,b}} \inner{C,P}$$
where $\inner{C,P} = \sum_{1 \le i,j\le n} C_{ij} P_{ij}$ is the usual inner-product for matrices and 
$$\Pi_{a,b} := \{ P \in \R^{n \times n}_{\ge 0} : P \mathbf{1} = a \text{ and } \mathbf{1}^T P = b \}.$$
Since $\Pi_{a,b}$ is a convex polytope, the minimum of $\inner{C,\cdot}$ is achieved at one of its extreme points.  Therefore the map $C \mapsto W_C(a,b)$ is a piecewise linear map.

For a fixed nonnegative matrix $A \in \R^{n \times n}$ with column sums equal to one, the \emph{Wasserstein distance map} is
$$\Phi_A(C) := W_C(a_i,a_j) + \tau \|C\|_\infty \|a_i - a_j\|_\infty$$
where $a_1, \ldots, a_n$ are the columns of $A$ and $\tau \ge 0$ is a regularization constant. It is easy to verify that $\Phi_A: \mathcal{D}_n \rightarrow \mathcal{D}_n$ is order-preserving and homogeneous.  When $\tau > 0$, the regularization term guarantees that $\Phi_A$ has a unique eigenvector in the interior of $\mathcal{D}_n$ \cite[Theorem 1]{HuCaPe21} (see also \cite[Lemma 6.2.1]{LemmensNussbaum}) and the normalized iterates of $\Phi_A$ will converge to the eigenvector at a linear rate. If we remove the regularization term by letting $\tau = 0$, then $\Phi_A$ may or may not still have an eigenvector in the interior $\inter{\mathcal{D}_n}$.  If it does, however, then note that $\Phi_A + \id$ is type K order-preserving and piecewise affine. Thus Corollaries \ref{cor:mainHomog} and \ref{cor:piecewiseHomog} imply that the iterates of the map $g(C) := \frac{\Phi_A(C)+C}{\|\Phi_A(C)+C\|_\infty}$ will converge at a linear rate to an eigenvector of $\Phi_A$ in $\inter{\mathcal{D}_n}$.  This partially explains an observation made in \cite[Section 3.3]{HuCaPe21} that numerical simulations seemed to converge to an eigenvector of $\Phi_A$ at a linear rate even when $\tau = 0$.  

\subsection{The H-eigenproblem for nonnegative tensors}
An \emph{order-$d$} \emph{tensor of dimension}-$n$ is an array of real numbers $\mathcal{A} = [\![a_{j_1\cdots j_d}]\!] \in \R^{n_1 \times \ldots \times n_d}$ where each $n_i = n$. A tensor is \emph{nonnegative} if each entry $a_{j_1 \cdots j_d}$ is nonnegative.  The H-eigenproblem for nonnegative tensors seeks to find eigenvectors $x \in \R^n_{\ge 0}$ and eigenvalues $\lambda \ge 0$ such that
$$\mathcal{A}x^{(d-1)} = \lambda x^{[d-1]}$$
where $x^{[d-1]} := (x_1^{d-1}, x_2^{d-1}, \ldots, x_n^{d-1})^T$ and 
$$(\mathcal{A}x^{(d-1)})_i := \sum_{1 \le j_2, \ldots, j_d \le n} a_{i  j_2 \cdots j_d} x_{j_2} \cdots x_{j_d}.$$
Of particular interest are eigenvectors with all positive entries.  In \cite[Theorem 5]{HuQi16}, Hu and Qi gave necessary and sufficient conditions for an order-$d$ nonnegative tensor of dimension-$n$ to have an eigenvector in $\R^n_{>0}$.  They refer to tensors with this property as \emph{strongly nonnegative}.  

For any order-$d$ nonnegative tensor $\mathcal{A}$ of dimension-$n$ let $f_\mathcal{A}: \R^n_{\ge 0} \rightarrow \R^n_{\ge 0}$ be the map 
$$f_\mathcal{A}(x) := (\mathcal{A}x^{(d-1)})^{[1/(d-1)]}.$$ 
Then $f_\mathcal{A}$ is order-preserving, homogeneous, and analytic, and any eigenvector of $f_\mathcal{A}$ is an H-eigenvector of the nonnegative tensor $\mathcal{A}$.  It is not immediately obvious, but the maps $f_\mathcal{A}$ are also multiplicatively convex (see e.g., \cite[Lemma 4.1]{Lins22}).  
If, in addition, there is a positive entry $a_{ij_2\cdots j_d}$ in $\mathcal{A}$ for every $i \in [n]$, then $f_{\mathcal{A}}(\R^n_{>0}) \subseteq \R^n_{>0}$.  Thus $\mathcal{A}$ is strongly nonnegative and $f_\mathcal{A}$ has an eigenvector in $\R^n_{>0}$ if and only if the basic and final classes of $f_\mathcal{A}$ are the same.   

For a strongly nonnegative tensor $\mathcal{A}$, the map $f_\mathcal{A} + \id$ is type K order-preserving. Therefore Theorem \ref{thm:anal} and Corollary \ref{cor:mainHomog} guarantee that the iterates of the map $g(x) = \frac{f_\mathcal{A}(x)+x}{\|f_\mathcal{A}(x)+x\|}$ converge to an H-eigenvector of $\mathcal{A}$ at a linear rate, no matter which starting point $x \in \R^n_{>0}$ is chosen.  This was already known for weakly irreducible nonnegative tensors (that is nonnegative tensors such that the graph $\G(f_\mathcal{A})$ is strongly connected). See \cite[Corollary 5.2]{FrGaHa13} for details.  See also \cite{ZhQi12,ZhQiWu13}.  Of course strongly nonnegative tensors are more general than weakly irreducible nonnegative tensors, and the entrywise positive H-eigenvectors of a strongly nonnegative tensor will not be unique (even after scaling) if the map $f_\mathcal{A}$ has more than one final class \cite[Theorem 7.1]{Lins22}. 

\subsection*{Acknowledgement} The author wishes to thank Roger Nussbaum his encouragement and helpful suggestions.  

\bibliography{DW2}
\bibliographystyle{plain}

\end{document}